\documentclass[12pt]{amsart}
\usepackage{graphicx,amsthm,amsmath,amsfonts,amssymb,epsfig}

\newtheorem{theorem}{Theorem}[section]
\theoremstyle{plain}

\newtheorem{claim}{Claim}[section]

\newtheorem{corollary}{Corollary}[section]

\newtheorem{definition}{Definition}[section]
\newtheorem{example}{Example}[section]

\newtheorem{lemma}{Lemma}[section]

\newtheorem{proposition}{Proposition}[section]

\numberwithin{equation}{section}

\begin{document}

\title{Numerical groups}
\author[m.t.Lozano]{Mar\'{\i}a Teresa Lozano}
\address[M.T. Lozano]{IUMA, Universidad de Zaragoza \\
Zaragoza, 50009, Spain}
\email[M.T. Lozano]{tlozano@unizar.es}
\thanks{Partially supported by grant MTM2016-76868-C2-2-P}
\author[J.M. Montesinos-Amilibia]{Jos\'{e} Mar\'{\i}a Montesinos-Amilibia}
\address[J.M. Montesinos]{Dto. Geometr\'{\i}a y Topolog\'{\i}a \\
Universidad Complutense, Madrid 28080 Spain}
\email[J.M. Montesinos]{jose{\_}montesinos@mat.ucm.es}
\date{\today }
\dedicatory{Dedicated to our friend Elena Mart\'{\i}n Peinador on the  occasion of her 70th birthday }
\begin{abstract}
A group of matrices $G$ with entries in a number field $K$ is defined to be \emph{numerical}
if $G$  has a finite  index subgroup of matrices whose
entries are algebraic integers. It is shown that an  irreducible or completely reducible subgroup of $GL(n,K)\subset GL(n,\mathbb{C})$  is
numerical if and only if the traces of its elements are algebraic integers. Some examples are given.
\end{abstract}
\maketitle

\section{Introduction}

Speaking loosely, two properties characterize an arithmetic group of
matrices with entries in a number field $K$. One is that it has a finite
index subgroup of matrices whose entries are algebraic integers (we call
such groups numerical). The second property is related to the number of
places of the field $K$ (see \cite{Vigneras1980}, \cite{HLM92} and \cite{EGM1991}) and
will not be considered here.

The purpose of this paper is to analyze the definition of numerical group and  characterize it in terms of
traces. This will be very useful when ascertaining if a given group of
is arithmetic or not.

The paper is organized as follows. In Section \ref{stf} we recall some known concepts on complex matrix subgroups of $GL(n,\mathbb{C})$. We prove for them that the condition of being irreducible or completely irreducible is characterized by the trace form (Theorem \ref{teorema10}).

In Section \ref{ng} we define the concept of numerical in a general framework, for subgroups of automorphisms of a vector space $V$ over $\mathbb{C}$. The main results are the following theorems

\noindent \textbf{Theorem  \ref{ifandonlyif}.}
\emph{Let $K$ be a number field and let $G$ be a irreducible or completely
reducible subgroup of $GL(n,K)\subset GL(n,\mathbb{C})$. Then $G$ is
numerical if and only if the trace of all the elements of $G$ are algebraic
integers. Moreover, $G$ is real numerical if $K$ is real.}

\noindent \textbf{Theorem  \ref{tcanume}.}
\emph{Let $G$ be a group of automorphisms of a vector space $V$ over $\mathbb{C}$
of dimension $n$. Then $G$ is a numerical group if and only if
there is a representation $\rho _{b}(G)$ such that $\rho _{b}(G)\subset GL(n,%
\mathcal{O}_{K}))$, where $\mathcal{O}_{K}$ is the ring of algebraic
integers in a number field $K$, which is real if $G$ is a real numerical
group.}

The proof of this Theorem is constructive and we will use this in a
forthcoming paper to revisit \cite{HLM92} and obtain the list of all the
arithmetic orbifols with the Borromean rings as the singular set in a new
way.

In the last section, Section \ref{sej}, we give some examples of numerical and no numerical subgroups in a family of subgroups of $SL(2,\mathbb{C})$ related to Hecke and Picard modular groups.

\section{The trace form on irreducible and completely irreducible subgroups of $GL(n,\mathbb{C})$}\label{stf}

Le $K$ be a field. Each element of the $K$-vector space $V=K^n$ is a $%
(n\times 1)$-matrix whose entries are in $K$.

Let $M(n,K)$ be the $K$-vector space of $(n\times n)$-matrices with entries
in $K$, and let $GL(n,K)$ be subset of invertible matrices (it has a group
structure under the matrix multiplication). There exits a left action of $%
M(n,K)$ on $V$ by matrix multiplication. These actions are the \emph{%
endomorphisms} of $V$. The actions of the elements of $GL(n,K)$ are the
\emph{automorphisms} of $V$.

Two matrices $A, B\in M(n,K) $ are \emph{similar} if there exists a matrix $%
U\in GL(n,K)$ such that $U^{-1}A U=B$. Because similar matrices represent
the same endomorphism with respect to (possibly) different bases, similar
matrices share all properties of the endomorphism that they represent, for
instance, the invariants derived from the characteristic polynomial, as
determinant and trace. Two groups $G, H\in GL(n,K) $ are \emph{equivalent}
if there exits a matrix $U\in GL(n,K)$ such that $U^{-1}G U=H$. We can also
say that two groups of endomorphism of $V$ are \emph{equivalent} if they are
represented by equivalent subgroups in $GL(n,K) $. (They are related by a
change of basis and therefore they share all algebraic properties as linear
groups.) The equivalence is used to work in a more convenient framework.

Given a subset $S\subset M(n,K)$, a subspace $W \subset V$ is call \emph{$S$%
- invariant space} if $A(W)\subset W$ for all $A\in S$. A $S$-invariant
space $W$ is \emph{minimal} if $\{ 0\}$ and $W$ are the only subspaces of $W$
which are $S$-invariant.

A subgroup $G$ of $GL(n,K)$ is called \emph{irreducible} if $V$ and $\{0\}$
are the only minimal $G$-invariant subspaces. A subgroup $G$ of $GL(n,K)$ is
called \emph{completely reducible} if there is a direct sum decomposition $%
V=\oplus _{i=1}^{r} V_i$ such that the restriction $G|_{V_i}$ is
irreducible, for each $i\in \{1,...,r\}$. In this case each $V_i$ is a
minimal $G$-invariant subspace.

Let $G$ be a subgroup of $GL(n,K)$, $W\subset V$ a minimal $G$-invariant
space of dimension $w$ and $b$ is a basis of $W$, the action of $G$ on $b$
defines a homomorphism
\begin{equation*}
\rho : G \longrightarrow \rho (G) \subset GL(w,K).
\end{equation*}
This map $\rho$ is an \emph{irreducible representation of $G$ of degree $w$}%
. Two representations $\rho$ and $\rho ^{\prime }$ are equivalent if there
exits a matrix $U\in GL(w,K)$ such that $U^{-1}\rho (A) U=\rho ^{\prime }(A)$
for each $A\in G$, both representations are related by a change of basis.

Let $\mathbb{C}$ denote the complex field. Let $G$ be a completely reducible
subgroup of $GL(n,\mathbb{C})$. Then $\mathbb{C}^n$ is a direct sum of
minimal $G$-invariant spaces. These subspaces can be ordered such that the
actions of $G$ on the first $k_1$ are equivalent; the actions of $G$ on the
next $k_2$ are also equivalent, and so on. There exists a matrix $U\in GL(n,%
\mathbb{C})$ such that for each $A\in G$ the matrix $U^{-1}AU$ is like this
\begin{equation*}
Diagonal[\overset{k_1}{\overbrace{A_1,...A_1}};...;\overset{k_1}{\overbrace{%
A_i,...A_i}};...;\overset{k_1}{\overbrace{A_s,...A_s}}]
\end{equation*}
where $\rho_i:A\longrightarrow A_i$ is an irreducible representation of
degree, say, $n_i$; $n= n_1k_1+...+n_sk_s$; $k_i$ is the multiplicity of $%
\rho _i$ and for $i\neq j$ the representations $\rho _i$ and $\rho _j$ are
nonequivalent. In this way we obtain a new basis in $\mathbb{C}^n$ and a
group $U^{-1}GU$ equivalent to $G$.

The \emph{linear hull} $\mathbb{C}[G]$ of $G$ is defined as the set of all
linear combinations $\left( \sum_i \lambda_i g_i \right)$, $\lambda_i \in
\mathbb{C}$, $g_i \in G$. It is the vector subspace on $M(n,\mathbb{C})$
generated by the matrices of $G$. $\mathbb{C}[G]$ is also a ring and an
associative algebra over $\mathbb{C}$ with unity $I$, where the product is
the multiplication of matrices
\begin{equation*}
\left( \sum_i \lambda_i g_i \right) \left( \sum_j \lambda_j g_j \right)=
\sum_{i,j} \lambda_i \lambda_j g_i g_j
\end{equation*}

The following theorem is a generalization of the Artin-Wedderburn Theorem
(See \cite[Section 14.4]{Ruso}).

\begin{theorem}
\label{AWT}  Let $G\subset GL(n,\mathbb{C})$ be a completely reducible
group. There exist integer positive numbers $n_{1},..,n_{s},k_{1},...,k_{s}$
and a matrix $U \subset GL(n,C)$ such that $U^{-1}\mathbb{C}[G]U$ consists
of all matrices of the following form
\begin{equation*}
Diagonal[A_{1},...,A_{1};...;A_{j},...,A_{j};...;A_{s},...,A_{s}]
\end{equation*}
where the multiplicity of $A_{i}$ is $k_{i}$ and
\begin{equation*}
A_{1},...,A_{s}
\end{equation*}
take independently all possible values in
\begin{equation*}
M(n_{1},\mathbb{C}),...,M(n_{s},\mathbb{C})
\end{equation*}
and $n=n_{1}k_{1}+...+n_{s}k_{s}$. In particular, if $G$ is irreducible then
$\mathbb{C}[G]=M(n,\mathbb{C})$. $_{\blacksquare}$
\end{theorem}

\begin{definition}
The \emph{trace form} $T$ is the bilineal symmetric form defined by
\begin{equation*}
\begin{array}{ccc}
T:M(n,\mathbb{C})\times M(n,\mathbb{C}) & \longrightarrow & \mathbb{C} \\
(A,B) & \to & T(A,B)=tr(AB)%
\end{array}
\end{equation*}
where $tr(A)$ denotes the trace of the matrix $A$. If $G\subset GL(n,\mathbb{%
C})$ is a group, the restriction of the trace form to $\mathbb{C}[G]$ is
denoted by $T_G$.
\end{definition}

The main result in this paragraph (Theorem \ref{teorema10}) is the
characterization of an irreducible or completely reducible group $G\subset
GL(n,\mathbb{C})$ by means of the trace form $T_G$.

\begin{claim}
The trace $T$ is a nondegenerate symmetric bilinear form on $M(n,\mathbb{C})$%
, that is, if $tr(AB)=0$ for all $B\in M(n,\mathbb{C})$ then $A=0$.
\end{claim}

\begin{proof}
Let $A=(a_{ij})$ be a matrix in $M(n,\mathbb{C})$ such that $tr(AB)=0$ for
all $B\in M(n,\mathbb{C})$. Denote by $E_{ij}$ the matrix with all entries $0
$ but the entry in position $ij$ which is $1$. The $n^{2}$ matrices $E_{ij}$
form a basis of the vector space $M(n,\mathbb{C})$. Therefore $%
tr(AE_{ij})=a_{ij}=0$ for all $i,j$, this implies $A=0$.
\end{proof}

An analogous demonstration will be used to prove the following result, which
is the first part of the looked for characterization.

\begin{theorem}
\label{teorema4}  Let $G$ be an irreducible or completely reducible subgroup
of $GL(n,\mathbb{C})$. Then $T_G$ is a nondegenerate symmetric bilinear form.
\end{theorem}

\begin{proof}
The vector space $Q$ of matrices
\begin{equation*}
Diagonal[A_{1},...,A_{1};...;A_{j},...,A_{j};...;A_{s},...,A_{s}]
\end{equation*}
considered in Theorem \ref{AWT} has dimension $n_1^2+...+n_s^2$.

We consider a new basis for $Q$ whose $n_1^2$ first vectors are
\begin{equation*}
F_{ij}^{1}=Diagonal[E_{ij},...,E_{ij};0,...,0;...;0,...,0]
\end{equation*}
where $i,j\in \{1,...,n_1\}$ . The next $n_{2}^{2}$ vectors in the new basis
of $Q$ are
\begin{equation*}
F_{ij}^{2}=Diagonal[0,...,0;E_{ij},...,E_{ij};0,...,0;...;0,...,0]
\end{equation*}
etc. Assume
\begin{equation*}
A=Diagonal[A_{1},...,A_{1};...;A_{j},...,A_{j};...;A_{s},...,A_{s}]
\end{equation*}
is such that $tr(AB)=0$ for all $B\in Q$. We denote $A_t=(a_{t,ij})$. Since $%
tr(AF_{ij}^t)=a_{t,ij}=0$, one obtains that $A=0$.
\end{proof}

The following example shows that there are groups for which the trace is a
degenerate form.

\begin{example}
Let $G \subset GL(n,\emph{C})$ be the group 
\begin{equation*}
G= \left\{ \left(
\begin{array}{cc}
1 & n \\
0 & 1
\end{array}
\right)\left| n\in \mathbb{Z} \right.\right\}
\end{equation*}
Then $\mathbb{C}[G]$ is the following dimension 2 algebra over $\mathbb{C}$
\begin{equation*}
\mathbb{C}[G]= \left\{ \left(
\begin{array}{cc}
x & y \\
0 & x
\end{array}
\right) \, | \, (x,y)\in \mathbb{C}^2 \right\}
\end{equation*}
and $T_G$ is degenerate because $T_G(X,A)=0$ for all $A\in \mathbb{C}[G]$
and
\begin{equation*}
X=\left(
\begin{array}{cc}
0 & 1 \\
0 & 0
\end{array}
\right).
\end{equation*}
\end{example}

To prove the second part of the characterization, that is the converse to
Theorem \ref{teorema4}, we need some results concerning the associative
algebras over a field. For the benefit of the reader we remember here some
necessary concepts. (See \cite[Chapter X]{Wedderburn}). Let $\mathcal{A}$ be
a finitely generated associative algebra over a field $K$. A subalgebra $%
\mathcal{B}\subset \mathcal{A}$ is an \emph{invariant subalgebra} if $%
\mathcal{AB}\leq \mathcal{B}$ and $\mathcal{BA}\leq \mathcal{B}$. When the
algebra has no unity, could be that $\mathcal{A}^2 < \mathcal{A}$, $\mathcal{%
A}^3< \mathcal{A}^2$, etc. Because $\mathcal{A}$ has finite dimension, for
some $n$
\begin{equation*}
\mathcal{A}^n< \mathcal{A}^{n-1}, \quad \mathcal{A}^{n+1} =\mathcal{A}^n
\end{equation*}
The integer $n$ is the \emph{index} of $\mathcal{A}$. If $\mathcal{A}^n=0$
the algebra is called \emph{nilpotent}. The maximal nilpotent invariant
subalgebra of $\mathcal{A}$ is the \emph{radical} $R(\mathcal{A})$ de $%
\mathcal{A}$. An algebra $\mathcal{A}$ which is not nilpotent and whose
radical $R(\mathcal{A})$ is $\{ 0\}$ is called \emph{semisimple}. If in
addition it has no invariant subalgebras it is said to be \emph{simple}.

The above concepts were generalizad to more general algebras and rings using
the \emph{Jacobson radical} of a ring $\mathcal{A}$, which is the ideal
whose elements annihilate all simple right $\mathcal{A}$-modules, but we do
not need here this further generalization.

The classification theorem of associative algebras over a field, \cite[10.10
Therem 5]{Wedderburn} contains the following characterization of semisimple
algebras.

\begin{theorem}
An algebra is semisimple if and only if can be expressed uniquely as the
direct sum of simple algebras. $_{\blacksquare}$
\end{theorem}

The following lemma relates minimal invariants subspaces for a group $%
G\subset GL(n,\mathbb{C})$ and for the corresponding algebra $\mathbb{C}%
[G]\subset Hom(V,V)\cong M(n,\mathbb{C})$.

\begin{lemma}
\label{lema7}  Let $G$ be a subgroup of $GL(n,\mathbb{C})$. A vector
subspace $W$ de $V=\mathbb{C}^n$ is a minimal $G$-invariant space if and
only if $W$ is a minimal $\mathbb{C}[G]$-invariant space.
\end{lemma}

\begin{proof}
Suppose that the subspace $W$ of $V$ is  $G$-invariant. Consider elements $%
x=\sum_i \lambda_i g_i \in \mathbb{C}[G]$ and $v\in W$. Then , since $W$ is
a vector space, we have
\begin{equation*}
x(v)= \left( \sum_i \lambda_i g_i\right) (v)= \sum_i \lambda_i g_i(v) \in W.
\end{equation*}
Reciprocally, if $x(v)\in W$ for every $x\in \mathbb{C}[G]$ and for every $%
v\in W$, because $G\subset \mathbb{C}[G]$, $W$ is $G$-invariante.

Therefore, it follows that a $G$-invariant subspace $W$ is minimal if and
only if it is minimal as $\mathbb{C}[G]$-invariant space.
\end{proof}

Then, the group $G\subset GL(n,\mathbb{C})$ is irreducible (resp. completely
reducible) if and only if $\mathbb{C}[G]$ is irreducible (resp. completely
reducible) as a vector subspace of $Hom(V,V)\cong M(n,\mathbb{C})$.

As a consequence of the above result (Lemma \ref{lema7}) and Theorem 3 in
\cite[p. 97]{Ruso}, we have

\begin{theorem}
A group $G \subset GL(n,\mathbb{C})$ is irreducible (completely reducible)
if and only if the algebra $\mathbb{C}[G]$ is simple (semisimple). $%
_{\blacksquare}$
\end{theorem}

\begin{theorem}
\label{9}  Let $G$ be a subgroup of $GL(n,\mathbb{C})$. If the trace form $%
T_G$ on the algebra $\mathbb{C}[G]$ is nondegenerate then the algebra $%
\mathbb{C}[G]$ is simple or semisimple.
\end{theorem}

\begin{proof}
Suppose the trace form $T_G$ is nondegenerate, and consider an element $Y$
in the radical $R(\mathbb{C}[G])$. The elements of the radical are nilpotent
matrices, and therefore all the eigenvalues of $Y$ are 0 and its trace is 0.
The radical $R(\mathbb{C}[G])$ is an invariant subalgebra of $\mathbb{C}[G]$%
, then $AY\in R(\mathbb{C}[G])$ for all $A\in \mathbb{C}[G]$. Therefore $%
T_G(AY)=0$ for all $A\in \mathbb{C}[G]$, and being the trace form $T_G$
nondegenerate we deduce that $Y=0$. Therefore $R(\mathbb{C}[G])=\{0\}$ and
the algebra $\mathbb{C}[G]$ is semisimple.
\end{proof}

Gathering together Theorem \ref{teorema4} and Theorem \ref{9} we obtain the
following characterization of completely reducible group by means of the
trace form that we could not find explicitly in the literature.

\begin{theorem}
\label{teorema10}  Let $G$ be a subgroup of $GL(n,\mathbb{C})$. The group $G$
is irreducible or completely reducible if and only if the trace form $T_G$
on the algebra $\mathbb{C}[G]$ is a nondegenerate form. $_\blacksquare$
\end{theorem}

\section{Numerical groups}\label{ng}

Let $V$ be a vector space of finite dimension $n$ over the field $\mathbb{C}$%
. Let $G$ be a group of automorphisms of $V$. For each basis $(b)$ in $V$,
the group $G$ is represented by $\rho_b(G)\subset GL(n,\mathbb{C})$, where
\begin{equation*}
\rho :G\longrightarrow GL(n,\mathbb{C})
\end{equation*}
is the representation obtained by expressing the elements of $G$ in the
basis $(b)$. All these subgroups
\begin{equation*}
\mathfrak{G}= \left\{\rho_b(G)\, | \, (b) \text{ basis of }\,V\right\} ,
\end{equation*}
form an equivalent class of subgroups of $GL(n,\mathbb{C})$. The equivalence
is given by the matrix associated to the change of basis. Therefore we can
study algebraic properties of $G$ by analyzing properties of elements of $%
\mathfrak{G}$. In particular arithmetic properties. In this section we
define the concept of numerical group.

\begin{definition}
Let $V$ be a vector space of finite dimension $n$ over the field $\mathbb{C}$%
. Let $G$ be a group of automorphisms of $V$. The group $G$ is a \underline{%
numerical group} if and only it has a representation $\rho_b(G)$ with the
following properties:

\begin{itemize}
\item $\rho_b(G)\subset GL(n,K)$, where $K$ is a number field

\item there exists a finite index subgroup $G_0 \subset G$ such that $%
\rho_b(G_0)\subset GL(n,\mathcal{O}_K)$, where $\mathcal{O}_K$ is the ring
of algebraic integers in the field $K$, and $GL(n,\mathcal{O}_K)$ is the set
of matrices with entries in $\mathcal{O}_K$ with determinant in the group of
units of the ring $\mathcal{O}_K$ ( \cite{N}, pg. 103)
\end{itemize}

If $K\subset \mathbb{R}$ (resp. $K\not\subset \mathbb{R}$) we say the the
numerical group is real (complex) numerical group.
\end{definition}

In orden to characterize numerical groups by using trace computations, we
next prove some lemmas.

\begin{lemma}
Let $G$ be a subgroup of $GL(n,\mathbb{C})$ such that the matrices of a
finite index subgroup $G_{o}\subset G$ have entries in the ring of integers $%
\mathcal{O}_K$ of a number field $K$. Let $L_{G}$ be the $\mathcal{O}_K$%
-module (submodule of $\mathcal{O}_K$-module $\mathbb{C}^{n}$) generated by
linear combinations with coefficients in $\mathcal{O}_K$ of the columns of
the elements of  $G$. Then:

\begin{enumerate}
\item $L_{G}$ is a finite generated free module, which contains the
canonical basis.

\item $G(L_{G})\subset L_{G}$
\end{enumerate}
\end{lemma}

\begin{proof}
Let $m$ be the index of $G_{o}$ in $G$. There exits $m$ elements $%
g_{1},...,g_{m}$ of $G$ such that $g_{1}$ is the identity matrix $I\!d$ and
\begin{equation*}
G=\bigcup _{i=1}^{m}g_{i}G_{o}
\end{equation*}
where $g_{i}G_{o}$ are the left cosets of $G_0$ in $G$. Let $L_{G}$ be the $%
\mathcal{O}_K$-module (submodule of $\mathcal{O}_K$-module $\mathbb{C}^{n}$)
generated by the columns of the matrices of  $G$. That is, an element of  $%
L_{G}$ is a linear combination with coefficients in $\mathcal{O}_K$ of
columns of elements of  $G$. To prove that $L_{G}$ is finitely generated
module, observe that since the entries of the elements of $G_{o}$ are in $%
\mathcal{O}_K$, each column of an element of $G_{o}$ is a linear combination
with coefficients in $\mathcal{O}_k$ of the columns $e_{1},...,e_{n}$ of the
identity matrix $I\!d$. Next we prove that $L_{G}$ is generated by the
columns of the elements $g_{1},...,g_{m}$. Indeed, the $i$-th column of $%
g_{k}\gamma \in g_{k}G_{o}$, $k\in \{1,...,m\}$, $\gamma \in G_{o}$, is the
product of the matrix $g_{k}$ by the $i$-th column $(\gamma _{1i},...,\gamma
_{ni})$ of the matrix $\gamma $. If $c_{1},...,c_{n}$ denote the $n$ columns
of $g_{k}$, then the $i$-th column of $g_{k}\gamma $ is $\gamma
_{1i}c_{1}+...+\gamma _{ni}c_{n}$. Since the numbers $\gamma
_{1i},...,\gamma _{ni}$ are in $\mathcal{O}_K$ we conclude that $L_{G}$ is
generated by the  $nm$ columns of $g_{1},...,g_{m}$.

Because $g_{1}$ is the identity matrix $I\!d$, the module $L_{G}$ contains
the columns $e_{1},...,e_{n}$ of the identity matrix $I\!d$. The module $L_G$
is free because the $\mathcal{O}_K$-module $\mathbb{C}^{n}$ has no torsion.

The columns of the matrix $g\in G$ are $g(e_{j})$, $j=1,...,n$. Therefore
the action of an element $g^{\prime }$ of $G$ in the column $g(e_{j})$:
\begin{equation*}
g^{\prime }(g(e_{j}))=g^{\prime }\circ g(e_{j})
\end{equation*}
is the column $g^{\prime }\circ g(e_{j})$ of $g\circ g\in G$ . This shows
that $G(L_{G})\subset L_{G}$.
\end{proof}

The difficulty in characterizing the numerical group property in terms of
traces is that the ring  $\mathcal{O}_{K}$ might not be principal. To
circumvent this problem we prove the following lemma. We denote by $%
[v_{1},...,v_{m}]_{\mathcal{O}}$ the submodule generated over the ring $%
\mathcal{O}$ by the vectors $v_{1},...,v_{m}$.

\begin{lemma}
Let $K\subset \mathbb{C}$ be a number field and let $\mathcal{O}_{K}$ be its
ring of integers. Let $L_{K}\subset K^{n}\subset \mathbb{C}^{n}$ be the $%
\mathcal{O}_{K}$-submodule $[v_{1},...,v_{m}]_{\mathcal{O}_{K}}$ , where $%
v_{1},...,v_{m}$ is a generating system of  $K^{n}$ over $K$. Then there
exists a number field $K^{\prime }\supset K$ (real if $K$ is real) such that
the $\mathcal{O}_{K^{\prime }}$-module
\begin{equation*}
L_{K^{\prime }}=[v_{1},...,v_{m}]_{I_{K^{\prime }}}\subset K^{\prime
n}\subset \mathbb{C}^{n}
\end{equation*}
is generated by a basis $w_{1},...,w_{n}$ of $K^{\prime n}$; that is
\begin{equation*}
L_{K^{\prime }}=[w_{1},...,w_{n}]_{\mathcal{O}_{K^{\prime }}}\subset
K^{\prime n}\subset \mathbb{C}^{n}
\end{equation*}
Besides, if $\mathcal{O}_{K}$ is a principal ideal ring then one can choose
$K^{\prime }=K.$
\end{lemma}

\begin{proof}
Among all the generating systems $\{v^{\prime }_{1},...,v^{\prime
}_{m^{\prime }}\}$ of $K^{n}$ such that $L_{K}=[v^{\prime
}_{1},...,v^{\prime }_{m^{\prime }}]_{\mathcal{O}_{K}}$ there is at least
one with minimal number of vectors, say $M$. If $M=n$ there is nothing to
prove. Assume that $M>n$. Consider the set $S$ of all those generating
systems of length $M$ of $K^{n}$. Let $\{v^{\prime }_{1},...,v^{\prime
}_{M}\}\in S$. Then the vectors $v^{\prime }_{1},...,v^{\prime }_{M}$ are
linearly dependent over $K$ and multiplying the linear dependence by a big
enough rational integer $N$ we can consider that they are linearly dependent
over $\mathcal{O}_{K}$. Then
\begin{equation}  \label{e1}
a_{1}v^{\prime }_{1}+...+a_{M}v^{\prime }_{M}=0
\end{equation}
where $a_{1},...,a_{M}$ are algebraic integers in $\mathcal{O}_{K}$ not all
null. Suppose that the system $\{v^{\prime }_{1},...,v^{\prime }_{M}\}$ of $%
K^{n}$ is choosen such that the relation (\ref{e1}) has the minimal number
of no null coefficients. (None of these coefficients is 1.) This number is
at least two because $K^{n}$ has no torsion and the vectors $v_{i}$ are
different from zero. Suppose that $a_{1}$ and $a_{2}$ are different from
zero. Let $K^{\prime }$ be the number field $K^{\prime }\supset K$ (real if $%
K$ is real) such that the ideal $\left\langle a_{1},a_{2}\right\rangle _{%
\mathcal{O}_{K^{\prime }}}$ is principal (see \cite[Theorem 9.10]{Stewart};
if $\mathcal{O}_{K} $ is a principal ideal ring then one can choose $%
K^{\prime }=K$). Let $d\neq 0$ be a number such that
\begin{equation*}
\left\langle a_{1},a_{2}\right\rangle _{\mathcal{O}_{K^{\prime
}}}=\left\langle d\right\rangle _{\mathcal{O}_{K^{\prime }}}
\end{equation*}
where $a_{1}=\tau d$, $a_{2}=\sigma d$ with $d$, $\tau $, $\sigma \in
\mathcal{O}_{K^{\prime }}$ different from zero. There exists $\alpha $, $%
\beta \in \mathcal{O}_{K^{\prime }}$ such that $\alpha a_{1}+\beta a_{2}=d$.

Then $\alpha \tau d+\beta \sigma d=d$, implies $\alpha \tau +\beta \sigma =1$%
.

Let's write
\begin{equation*}
\left[
\begin{array}{cc}
\tau & \sigma \\
-\beta & \alpha%
\end{array}%
\right] \left[
\begin{array}{c}
v_{1} \\
v_{2}%
\end{array}%
\right] =\left[
\begin{array}{c}
u_{1} \\
u_{2}%
\end{array}%
\right]
\end{equation*}%
Then the $\mathcal{O}_{K^{\prime }}$-module
\begin{equation*}
L_{K^{\prime }}=[v^{\prime }_{1},...,v^{\prime }_{M}]_{\mathcal{O}%
_{K^{\prime }}}\subset K^{\prime n}
\end{equation*}%
is equal to the $\mathcal{O}_{K^{\prime }}$-module
\begin{equation*}
\lbrack u_{1},u_{2},v^{\prime }_{3},...,v^{\prime }_{M}]_{\mathcal{O}%
_{K^{\prime }}}\subset K^{\prime n}
\end{equation*}%
because the matrix
\begin{equation*}
\left[
\begin{array}{cc}
\alpha & -\sigma \\
\beta & \tau%
\end{array}%
\right]
\end{equation*}%
with entries in $\mathcal{O}_{K^{\prime }}$ is unimodular. Also
\begin{equation*}
\lbrack u_{1},u_{2},v^{\prime }_{3},...,v^{\prime }_{M}]_{K^{\prime
}}=[v^{\prime }_{1},...,v^{\prime }_{M}]_{K^{\prime }}=K^{\prime n}
\end{equation*}%
The linear dependence (\ref{e1}) is now written as a linear dependence in $%
L_{K^{\prime }}$:
\begin{equation*}
a_{1}(\alpha u_{1}-\sigma u_{2})+a_{2}(\beta u_{1}+\tau
u_{2})+a_{3}v^{\prime }_{3}+...+a_{M}v^{\prime }_{M}=0
\end{equation*}%
that is
\begin{equation*}
u_{1}(\alpha a_{1}+a_{2}\beta )+u_{2}(-\sigma a_{1}+\tau
a_{2})+a_{3}v^{\prime }_{3}+...+a_{M}v^{\prime }_{M}=0
\end{equation*}
And because
\begin{equation*}
-\sigma a_{1}+\tau a_{2}=(-\sigma \tau +\tau \sigma )d=0
\end{equation*}
we have
\begin{equation}  \label{e2}
du_{1}+a_{3}v^{\prime }_{3}+...+a_{M}v^{\prime }_{M}=0
\end{equation}
If $K^{\prime }=K$ then $\{u_{1},u_2,v^{\prime }_3,...,v^{\prime }_{M}\}\in S
$, but this is not possible because the relation (\ref{e2}) has less
coefficients different from zero that the relation (\ref{e1}). Therefore if $%
\mathcal{O}_{K}$ is a principal ideal ring necessarily $M=n$ and the lemma
is proven.

If $\mathcal{O}_{K}$ is not a principal ideal ring  and if all the
coefficients in (\ref{e2}) are multiple of one of them, a generator can be
eliminated. If this is not the case, the argument can be repeated decreasing
the number of non-zero coefficients. If the number of non-zero coefficients
is reduced to two necessarily a coefficient is multiple of the other and
thus a generator can be eliminated. At the final step one have $M=n$.
\end{proof}

\begin{proposition}
\label{pro}  Let $G$ be a subgroup of $GL(n,K)$ where $K$ is a number field.
Suppose that $G$ leaves invariant a $\mathcal{O}_{K}$-submodule $%
L_{K}=[v_{1},...,v_{m}]_{\mathcal{O}_{K}}$ of $K^{n}$, where $\{
v_{1},...,v_{m} \}$ is a generating system of $K^{n}$ over $K$. Then there
exists a number field $K^{\prime }$ (real if $K$ is real) containing $K$
such that the $\mathcal{O}_{K^{\prime }}$-submodule $L_{K^{\prime
}}=[v_{1},...,v_{m}]_{\mathcal{O}_{K^{\prime }}}$ of $K^{\prime n}$ is
generated by a basis of $K^{\prime n}$ and $G(L_{K^{\prime }})\subset
L_{K^{\prime }}$. Furthermore, if $\mathcal{O}_{K}$ is a principal ideal
ring then $K^{\prime }=K.$
\end{proposition}

\begin{proof}
The new number field $K^{\prime }$ is the given by the above lemma. For all
element $g\in G$
\begin{equation*}
g(v_{i})=\sum _{j=1}^{m}\lambda _{j}v_{j},\lambda _{j}\in \mathcal{O}%
_{K}\subset \mathcal{O}_{K^{\prime }}
\end{equation*}%
and $L_{K^{\prime }}$ is generated by $\{ v_{1},...,v_{m}\}$ over $\mathcal{O%
}_{K^{\prime }}$ therefore $g(L_{K^{\prime }})\subset L_{K^{\prime }}$.
\end{proof}

\begin{theorem}
\label{tcanume} Let $G$ be a group of automorphisms of a vector space $V$
over $\mathbb{C}$ of dimension $n$. Then $G$ is a \underline{numerical group}
if and only if there is a representation $\rho _{b}(G)$ such that $\rho
_{b}(G)\subset GL(n,\mathcal{O}_{K}))$, where $\mathcal{O}_{K}$ is the ring
of algebraic integers in a number field $K$, which is real if $G$ is a real
numerical group.
\end{theorem}

\begin{proof}
Let $G$ be a numerical group. Suppose $G$ is a subgroup of $GL(n,K)\subset
GL(n, \mathbb{C})$ and $G_o \subset GL(n,\mathcal{O}_K))$ is a finite index
subgroup of $G$. Let $L_{K}$ be the $\mathcal{O}_K$-module (submodule of the
$\mathcal{O}$-module $K^{n}$) generated by linear combinations of the
columns of the elements of $G$ with coefficients in $\mathcal{O}_K$. Then $%
L_{K}$ is a finitely generated free module containing the canonical basis
and $G(L_{K})\subset L_{K}$. By Proposition \ref{pro} there exists a number
field $K^{\prime }\supset K$ (real if $K$ is real) such that the $\mathcal{O}%
_{K^{\prime }}$-module
\begin{equation*}
L_{K^{\prime }}\subset K^{\prime n}\subset \mathbb{C}^{n},
\end{equation*}
generated by linear combinations of the columns of the elements of $G$ with
coefficients in $\mathcal{O}_{K^{\prime }}$, is generated by a basis $\{
w_{1},...,w_{n}\}$ of $K^{\prime n}$; that is
\begin{equation*}
L_{K^{\prime }}=[w_{1},...,w_{n}]_{\mathcal{O}_{K^{\prime }}}\subset
K^{\prime n}\subset \mathbb{C}^{n}
\end{equation*}
Furthemore, $G(L_{K^{\prime }})\subset L_{K^{\prime }}$. Then writing $g\in G
$ in the basis $\{ w_{1},...,w_{n}\}$ a matrix with entries in $\mathcal{O}%
_{K^{\prime }}$ is obtained.
\end{proof}

As a consequence, we obtain the following criterion, that give us a useful
test to prove that a group of automorphism of a vector space is not
numerical.

\begin{corollary}
\label{ifpart}  If a group $G$ of automorphism of a vector space $V$ is
numerical then the trace of all the elements of $G$ are algebraic integers
of a number field.
\end{corollary}

The inverse of this result is true under the assumption that the group $G$
is irreducible or completely reducible.

\begin{theorem}
\label{ifandonlyif} Let $K$ be a number field and let $G$ be a irreducible
or completely reducible subgroup of $GL(n,K)\subset GL(n,\mathbb{C})$. Then $%
G$ is numerical if and only if the trace of all the elements of $G$ are
algebraic integers. Moreover, $G$ is real numerical if $K$ is real.
\end{theorem}

\begin{proof}
Suppose $G$ is an irreducible or completely reducible subgroup of $%
GL(n,K)\subset GL(n,\mathbb{C})$ and that the trace of all the elements of $G
$ are elements of $\mathcal{O}_{K}$. Recall that by Theorem \ref{teorema4},
the trace form $T$ is a nondegenerate bilinear form on $\mathbb{C}[G]$.

Let $\{ g_{1},...,g_{r}\} \in G$ be a basis of the $\mathbb{C}$-vector space
$\mathbb{C}[G]$:
\begin{equation*}
\lbrack g_{1},...,g_{r}]_{\mathbb{C}}=\mathbb{C}[G].
\end{equation*}%
Let $\{ g_{1}^{\ast },...,g_{r}^{\ast }\in \mathbb{C}[G]\} $ be the dual
basis of $\{ g_{1},...,g_{r}\}$ with respect to the bilinear form $T$. That
is $T(g_{i}^{\ast },g_{j})=\delta _{ij}$.

\textbf{Step 1: $\mathcal{O}_{K}[G]$.} Let $\mathcal{O}_{K}[G]$ be the $%
\mathcal{O}_{K}$-module generated by $G$ with coefficients in $\mathcal{O}%
_{K}$. We next prove that $\mathcal{O}_{K}[G]$ is finitely generated. The
module $[g_{1},...,g_{r}]_{\mathcal{O}_{K}}$ is a submodule of $\mathcal{O}%
_{K}[G]$. On the other hand, if  $g\in G$ we can write
\begin{equation}  \label{eg}
g=a_{1}g_{1}^{\ast }+...+a_{r}g_{r}^{\ast }
\end{equation}
where $a_{i}\in \mathbb{C}$. But
\begin{equation}  \label{tggj}
T(g,g_{j})=T(a_{1}g_{1}^{\ast }+...+a_{r}g_{r}^{\ast
},g_{j})=a_{j}=tr(gg_{j})\in \mathcal{O}_{K}
\end{equation}
by hypothesis. Therefore $\mathcal{O}_{K}[G]$ is a submodule of $%
[g_{1}^{\ast },...,g_{r}^{\ast }]_{\mathcal{O}_{K}}$:
\begin{equation*}
\lbrack g_{1},...,g_{r}]_{\mathcal{O}_{K}}\leq \mathcal{O}_{K}[G]\leq
\lbrack g_{1}^{\ast },...,g_{r}^{\ast }]_{\mathcal{O}_{K}}
\end{equation*}%
By applying (\ref{eg}) and (\ref{tggj}) to $g_i$, we have $%
g_{i}=\sum_{j=1}^{r}tr(g_{i}g_{j})g_{j}^{\ast }$. The matrix $%
(tr(g_{i}g_{j}))$ has entries in $\mathcal{O}_{K}$ ant its determinant is
different from zero because it corresponds to a change of basis in  $\mathbb{%
C}[G]$. Therefore $g_{i}^{\ast }=\sum_{j=1}^{r}b_{ij}g_{j}$ where $(b_{ij})$
is the inverse matrix of $(tr(g_{i}g_{j}))$. Hence $b_{ij}\in K$. But then
there exists a rational integer $N\geq 1$ such that $Nb_{ij}\in \mathcal{O}%
_{K}$ for all $b_{ij}$ and then $g_{i}^{\ast
}=\sum_{i=1}^{r}(Nb_{ij})(g_{j}/N)$. In other words
\begin{equation*}
C:=[g_{1},...,g_{r}]_{\mathcal{O}_{K}}\leq \mathcal{O}_{K}[G]\leq \lbrack
g_{1}^{\ast },...,g_{r}^{\ast }]_{\mathcal{O}_{K}}\leq \lbrack \frac{g_{1}}{N%
},...,\frac{g_{r}}{N}]_{\mathcal{O}_{K}}:=B
\end{equation*}
Being $\mathcal{O}_{K}$ the ring of integers of a number field, it is a
finitely generated abelian group. Let $\{ \zeta _{1},...,\zeta _{s}\}$ be a
generating system. Then a finitely generated module over $\mathcal{O}_{K}$,
as  $C$ (or $B$), is a finitely generated abelian group generated by $\zeta
_{i}g_{j}$ (resp. $\zeta _{i}\frac{g_{j}}{N}$). The quotient group $B/C$ is
finite because it is finitely generated and every element has finite order:
Observe that if
\begin{equation*}
b=\sum_{j=1}^{r}\varepsilon _{j}(g_{j}/N)\in B,\varepsilon _{j}\in \mathcal{O%
}_{K},
\end{equation*}%
then $Nb\in C$. Therefore the abelian quotient group $\mathcal{O}_{K}[G]/C$
is finite.

Let $[h_{1}],...,[h_{t}]$ be the elements of $\mathcal{O}_{K}[G]/C$. Then if
$a\in\, \mathcal{O}_{K}[G]$, the element $[a]$ of $\mathcal{O}_{K}[G]/C$ is
written as $[a]=m_{1}[h_{1}]+...+m_{t}[h_{t}]$ with rational integer
coefficients. Therefore  $a=m_{1}h_{1}+...+m_{t}h_{t}+h$ where $h\in C$.
Then $a=m_{1}h_{1}+...+m_{t}h_{t}+\alpha _{1}g_{1}+...+\alpha _{r}g_{r}$, $%
\alpha _{i}\in \mathcal{O}_{K}$. Thus
\begin{equation*}
\mathcal{O}_{K}[G]=[h_{1},...,h_{t},g_{1},...,g_{r}]_{\mathcal{O}_{K}}
\end{equation*}%
is finitely generated. Let's write for short $\mathcal{O}_{K}[G]=[\gamma
_{1},...,\gamma _{s}]_{\mathcal{O}_{K}}$ where $\gamma _{j}$ is a linear
combination of matrices of $G$ with coefficients in $\mathcal{O}_{K}$.

\textbf{Step 2: $L_{G}$.} The module $L_{G}^{i}$ generated over $\mathcal{O}%
_{K}$ by the $i$-th column of the matrices of $G$ is finitely generated over
$\mathcal{O}_{K}$ by the $i$-th columns of the matrices $\gamma
_{1},...,\gamma _{s}$ which are a linear combination with coefficients in $%
\mathcal{O}_{K}$ of the  $i$-th columns $g_{1}^{i},...,g_{r}^{i}$ of the
matrices $g_{1},...,g_{r} $. Therefore the module $L_{G}^{i}$ is generated
over $\mathcal{O}_{K}$ by the columns $g_{1}^{i},...,g_{r}^{i}$.

An element of the module $L_{G}$ generated over $\mathcal{O}_{K}$ by the
columns of the matrices of $G$ is a linear combination with coefficients in $%
\mathcal{O}_{K}$ of elements in $L_{G}^{1},...,L_{G}^{n}$. Therefore $L_{G}$
is finitely generated over $\mathcal{O}_{K}$ by the columns of $%
g_{1},...,g_{r}$. The module $L_{G}$ contains the canonical basis (because $%
I\!d\in G$) and $G(L_{G})\subset L_{G}$. We write $L_{G}=[\lambda
_{1},...,\lambda _{u}]_{\mathcal{O}_{K}}$ where the vectors $\lambda _{i}$
are in $K^{n}$, since they are columns of $G$ which is a subgroup of $GL(n,K)
$.

\textbf{Step 3: $\rho_b(G)\subset GL(n,\mathcal{O}_{K^{\prime }})$}

By Proposition \ref{pro} there exists a number field $K^{\prime }$ (real if $%
K$ is real)containing $K$ such that the $\mathcal{O}_{K^{\prime }}$%
-submodule $L_{K^{\prime }}=[\lambda _{1},...,\lambda _{u}]_{\mathcal{O}%
_{K^{\prime }}}$ of $K^{\prime n}$ is generated by a basis $(b)=\{
w_{1},...,w_{n}\}$ of $K^{\prime n}$ and $G(L_{K^{\prime }})\subset
L_{K^{\prime }}$. The system $\{ w_{1},...,w_{n}\}$ generates over $\mathbb{C%
}$ the space $\mathbb{C}^{n}$ because $K^{n}$ contains the canonical basis.
Therefore $(b)$ is a basis of $\mathbb{C}^{n}$. The matrices of the
representation $\rho_b(G)$ of $G$ in the basis $(b)$, have its entries in $%
\mathcal{O}_{K^{\prime }}$ because $G(L_{K^{\prime }})\subset L_{K^{\prime }}
$. Then Theorem \ref{tcanume} implies that $G$ is a numerical group, which
is real numerical if $K$ is real.
\end{proof}

\section{Some numerical subgroups of $SL(2,\mathbb{C})$}\label{sej}

The study of numerical subgroups of $SL(2,\mathbb{C})$ is interesting in order to find discrete groups of isometries of the hyperbolic space $\mathbb{H}^3$. Recall that the group $PSL(2,\mathbb{C})$ is isomorphic to  $Iso^+(\mathbb{H}^3)$, the  group of  orientation preserving isometries of the hyperbolic space $\mathbb{H}^3$ in the half space Poincar\'{e} model. Here we analize a family of groups which contain groups  related to some important ones as the Modular group, Hecke groups and Picard and Bianchi groups.

Consider the matrix group $\Pi (\lambda)$, $\lambda \in \mathbb{C}$, generated by the two matrices
\[
A=\left(
    \begin{array}{cc}
      0 & -1 \\
    1 & 0 \\
    \end{array}
  \right), \quad
 B=\left(
    \begin{array}{cc}
      1 & \lambda \\
    0 & 1 \\
    \end{array}
  \right).
\]  Observe that $\Pi (\lambda) \subset SL(2,\mathbb{Z}(\lambda ))\subset SL(2,\mathbb{C})$, where $\mathbb{Z}(\lambda)$ is the polynomial ring over $\mathbb{Z}$ with $\lambda$ as indeterminate.

\begin{proposition}
  The group $\Pi (\lambda)$ is a numerical group if and only if $\lambda$ is an algebraic integer.
\end{proposition}

\begin{proof}
The group $\Pi (\lambda)$ is an irreducible subgroup of $SL(2,\mathbb{Z}(\lambda ))$. Then, according to Theorem \ref{ifandonlyif} it is numerical if and only if the trace of all the elements of $\Pi (\lambda)$ are algebraic integers.
  In $SL(2, \mathbb{C})$ there exists the following useful result (\cite{Vo}, \cite{Ma}, \cite{CS}, \cite{GAMA1992} \cite{HLM95}):
Sea $H\subset SL(2,\mathbb{C})$. The trace $t_h$ of every element $h\in H$
can be written as a polynomial with rational coefficients in the variables
\begin{equation*}
\begin{array}{ll}
t_{h_i}, & 1\leq i\leq n \\
t_{h_ih_j}, & 1\leq i<j\leq n \\
t_{h_ih_jh_k}, & 1\leq i<j<k\leq n%
\end{array}
\end{equation*}
where $\{h_1,...,h_n\}$ is any set of generators of $H$.

The group $\Pi (\lambda)$ is generated
by two elements $\{A,B\}$, then for $h\in \Pi (\lambda)$,  $t_h$ can be written as a
polynomial with integer coefficients in the variables
$\{t_{A},t_{B},t_{AB}\}$.
\[
t_A = 0,\quad t_B= 2,\quad t_{AB}= \lambda
\]
 Therefore $t_h$ is an algebraic integer if and only if $\lambda$ is an algebraic integer.
\end{proof}

For real $\lambda$ the group $\Pi (\lambda)$ projects on $PSL(2,\mathbb{R})$ to the Hecke group $H(\lambda)$. Hecke groups were introduced by Hecke \cite{Hecke36}. Recall that $PSL(2,\mathbb{R})$ is the group of  orientation preserving isometries of the hyperbolic plane $\mathbb{H}^2$ in the upper half-plane model, acting as M\"{o}ebius transformations.
\[
\pm \left(
      \begin{array}{cc}
        a & b \\
        c & d \\
      \end{array}
    \right)\in PSL(2,\mathbb{R}) \quad \longrightarrow \quad z\to \frac{az+b}{cz+d}
\]

The Hecke group $H(\lambda)$ is the Fuchsian group generated by the two M\"{o}ebius transformations
\[
A(z)= -\frac{1}{z} \quad \text{and}\quad B(z)=z+\lambda
\]
$H(1)$ is the Modular group.
When $\lambda > 2$, the Hecke group $H(\lambda)$ is of second kind; that is, the fundamental domain in the hyperbolic plane has infinite volume.  Hecke proved that $H(\lambda)$ is discrete if and only if $\lambda \geq 2$ or $\lambda = \lambda_q=2\cos (\pi /q)$, $q\in \mathbb{N}$, $q\geq 3$. Then, here we have examples of discrete subgroups of $SL(2,\mathbb{R})$ which are not numerical groups: $\Pi (\lambda)$ with $\lambda > 2$ and $\lambda$ no an algebraic integer. The action on the hyperbolic plane $\mathbb{H}^2$ of the projection of these groups on $PSL(2,\mathbb{R})$ produces a hyperbolic orbifold surface with infinite volume.

The numerical group  $\Pi (\sqrt{-d})$ projects on $PSL(2,\mathbb{C})$ to a subgroup of the Bianchi group $PSL(2,\mathcal{O}_d)$, where $\mathcal{O}_d$ denote the ring of integers in $\mathbb{Q}(\sqrt{-d})$. In particular, the numerical group $\Pi (i)$ projects on $PSL(2,\mathbb{C})$ to a subgroup of the Picard modular group $PSL(2,\mathcal{O}_1)$. Those groups have been widely studied, see for instance references in \cite{MR1991}.

\end{document}